\documentclass[reqno]{amsart}
\usepackage{amssymb,mathtools}%,lineno

\usepackage{ifpdf}
\ifpdf
 \usepackage[hyperindex]{hyperref}
\else
 \expandafter\ifx\csname dvipdfm\endcsname\relax
 \usepackage[hypertex,hyperindex]{hyperref}%
 \else
 \usepackage[dvipdfm,hyperindex]{hyperref}%
 \fi
\fi
\allowdisplaybreaks[4]
\theoremstyle{plain}
\newtheorem{thm}{Theorem}
\theoremstyle{remark}
\newtheorem{rem}{Remark}

\DeclareMathOperator{\td}{d\!}
\DeclareMathOperator{\te}{e}
\DeclareMathOperator{\ti}{i}
\DeclareMathOperator{\bell}{B}

\begin{document}

\title[Uniform treatments of Bernoulli and Stirling numbers]
{Uniform treatments of Bernoulli numbers, Stirling numbers, and their generating functions}

\author[F. Qi]{Feng Qi}
\address{School of Mathematics and Informatics, Henan Polytechnic University, Jiaozuo 454010, Henan, China;
Retired researcher, 17709 Sabal Court, Dallas, TX 75252-8024, USA;
School of Mathematics and Physics, Hulunbuir University, Hailar 021008, Inner Mongolia, China}
\email{\href{mailto: F. Qi <qifeng618@gmail.com>}{qifeng618@gmail.com}}
\urladdr{\url{https://orcid.org/0000-0001-6239-2968}}

\begin{abstract}
In this paper, by virtue of a determinantal formula for derivatives of the ratio between two differentiable functions, in view of the Fa\`a di Bruno formula, and with the help of several identities and closed-form formulas for the partial Bell polynomials $\bell_{n,k}$, the author
\begin{enumerate}
\item
establishes thirteen Maclaurin series expansions of the functions
\begin{align*}
&\ln\frac{\te^x+1}{2}, && \ln\frac{\te^x-1}{x}, && \ln\cosh x, \\ 
&\ln\frac{\sinh x}{x}, && \biggl[\frac{\ln(1+x)}{x}\biggr]^r, && \biggl(\frac{\te^x-1}{x}\biggr)^r
\end{align*}
for $r=\pm\frac{1}{2}$ and $r\in\mathbb{R}$ in terms of the Dirichlet eta function $\eta(1-2k)$, the Riemann zeta function $\zeta(1-2k)$, and the Stirling numbers of the first and second kinds $s(n,k)$ and $S(n,k)$.
\item
presents four determinantal expressions and three recursive relations for the Bernoulli numbers $B_{2n}$.
\item
finds out three closed-form formulas for the Bernoulli numbers $B_{2n}$ and the generalized Bernoulli numbers $B_n^{(r)}$ in terms of the Stirling numbers of the second kind $S(n,k)$, and deduce two combinatorial identities for the Stirling numbers of the second kind $S(n,k)$.
\item
acquires two combinatorial identities, which can be regarded as diagonal recursive relations, involving the Stirling numbers of the first and second kinds $s(n,k)$ and $S(n,k)$.
\item
recovers an integral representation and a closed-form formula, and establish an alternative explicit and closed-form formula, for the Bernoulli numbers of the second kind $b_n$ in terms of the Stirling numbers of the first kind $s(n,k)$.
\item
obtains three identities connecting the Stirling numbers of the first and second kinds $s(n,k)$ and $S(n,k)$.
\end{enumerate}
\par
The most highlights of this paper include the unification $\bigl(\frac{\te^x-1}{x}\bigr)^r$ of the generating functions of the Bernoulli numbers $B_n$ and the Stirling numbers of the second kind $S(n,k)$, the unification $\bigl[\frac{\ln(1+x)}{x}\bigr]^r$ of the generating functions of the Bernoulli numbers of the second kind $b_n$ and the Stirling numbers of the first kind $s(n,k)$, and the disclosure of the transformations between these two unifications.
\end{abstract}



%\thanks{*Corresponding author}

\thanks{This paper was typeset using \AmS-\LaTeX}

\maketitle
\tableofcontents

\section{Preliminaries}
According to~\cite[Fact~13.3]{Bernstein=2018-MatrixMath}, for $z\in\mathbb{C}$ such that $\Re(z)>1$, the Riemann zeta function $\zeta(z)$ can be defined by
\begin{equation}\label{zeta-three-def-eq}
\zeta(z)=\sum_{k=1}^{\infty}\frac{1}{k^z}
=\frac{1}{1-2^{1-z}}\sum_{k=1}^{\infty}\frac{(-1)^{k-1}}{k^z}
=\frac{1}{1-2^{1-z}}\eta(z),
\end{equation}
where $\eta(z)$ is called the Dirichlet eta function.
In~\cite[Section~3.5, pp.~57--58]{Temme-96-book}, the Riemann zeta function $\zeta(z)$ is analytically extended from $\Re(z)>1$ to the punctured complex plane $\mathbb{C}\setminus\{1\}$ such that the only singularity $z=1$ is a simple pole with residue $1$. In other words, the Riemann zeta function $\zeta(z)$ is meromorphic with a simple pole at $z=1$. Consequently, by virtue of the relation~\eqref{zeta-three-def-eq}, the Dirichlet eta function
\begin{equation}\label{eta-zeta-equal}
\eta(z)=\bigl(1-2^{1-z}\bigr)\zeta(z)
\end{equation}
can be continued as an entire function in $z\in\mathbb{C}$. See also~\cite[Chapter~6]{Bateman-Diamond-B2004} and~\cite{Mon-Eta-Ratio.tex}.
\par
It is common knowledge that the Bernoulli numbers $B_n$ are generated by
\begin{equation}\label{Bernoulli-Gen-Eq}
\frac{z}{\te^z-1}=\sum_{n=0}^\infty B_n\frac{z^n}{n!}=1-\frac{z}2+\sum_{n=1}^\infty B_{2n}\frac{z^{2n}}{(2n)!}, \quad |z|<2\pi;
\end{equation}
that the Stirling numbers of the second kind $S(n,k)$ for $n\ge k\ge0$ can be analytically generated~\cite[pp.~131--132]{Quaintance-Gould-2016-B} by
\begin{equation}\label{2Stirl-funct-rew}
\biggl(\frac{\te^z-1}{z}\biggr)^k=\sum_{n=0}^\infty \frac{S(n+k,k)}{\binom{n+k}{k}} \frac{z^{n}}{n!}, \quad k\ge0;
\end{equation}
and that the Stirling numbers of the first kind $s(n,k)$ for $n\ge k\ge0$ can be analytically generated~\cite[Theorem~3.14]{Mansour-Schork-B2016} by
\begin{equation}\label{Stirl-No-First-GF}
\biggl[\frac{\ln(1+z)}{z}\biggr]^k=\sum_{n=0}^\infty \frac{s(n+k,k)}{\binom{n+k}{k}}\frac{z^{n}}{n!}, \quad |z|<1.
\end{equation}
\par
In~\cite[p.~40, Entry~5]{Bourbaki-Spain-2004}, we find a general derivative formula
\begin{equation}\label{Sitnik-Bourbaki-reform}
\hskip-8em\frac{\operatorname{d}^k}{\td x^k}\biggl[\frac{p(x)}{q(x)}\biggr]
=\frac{(-1)^k}{q^{k+1}(x)}
\begin{vmatrix}
p(x) & q(x) & 0 &\dotsm & 0 & 0\\
p'(x) & q'(x) & q(x) &\dotsm& 0 & 0\\
p''(x) & q''(x) & \binom{2}{1}q'(x) &\dotsm & 0 &0\\
\vdots & \vdots & \vdots & \ddots & \vdots & \vdots\\
p^{(k-2)}(x) & q^{(k-2)}(x) & \binom{k-2}{1}q^{(k-3)}(x) & \dotsm & q(x) & 0\\
p^{(k-1)}(x) & q^{(k-1)}(x) & \binom{k-1}{1}q^{(k-2)}(x) & \dotsm & \binom{k-1}{k-2}q'(x) & q(x)\\
p^{(k)}(x) & q^{(k)}(x) & \binom{k}{1}q^{(k-1)}(x) & \dotsm & \binom{k}{k-2}q''(x) & \binom{k}{k-1}q'(x)
\end{vmatrix}
\end{equation}
for $k\in\mathbb{N}_0$, where the functions $p(x)$ and $q(x)$ are of the $k$th derivatives.
Let $H_0=1$ and
\begin{equation*}
H_k=
\begin{vmatrix}
h_{1,1} & h_{1,2} & 0 & \dotsc & 0 & 0\\
h_{2,1} & h_{2,2} & h_{2,3} & \dotsc & 0 & 0\\
h_{3,1} & h_{3,2} & h_{3,3} & \dotsc & 0 & 0\\
\vdots & \vdots & \vdots & \ddots & \vdots & \vdots\\
h_{k-2,1} & h_{k-2,2} & h_{k-2,3} & \dotsc & h_{k-2,k-1} & 0 \\
h_{k-1,1} & h_{k-1,2} & h_{k-1,3} & \dotsc & h_{k-1,k-1} & h_{k-1,k}\\
h_{k,1} & h_{k,2} & h_{k,3} & \dotsc & h_{k,k-1} & h_{k,k}
\end{vmatrix}
\end{equation*}
for $k\in\mathbb{N}$. Theorem in~\cite[p.~222]{Cahill-Narayan-Fibonacci-2004} states that the sequence $H_k$ for $k\in\mathbb{N}_0$, with the assumption $H_1=h_{1,1}$, satisfies the recursive relation
\begin{equation}\label{Cahill-Narayan-Fibonacci-2004-Thm}
H_k=\sum_{\ell=1}^k(-1)^{k-\ell}h_{k,\ell} \Biggl(\prod_{j=\ell}^{k-1}h_{j,j+1}\Biggr) H_{\ell-1}
\end{equation}
for $k\ge2$, where an empty product is understood to be $1$. The derivative formula~\eqref{Sitnik-Bourbaki-reform} for the ratio of two differentiable functions and the recursive relation~\eqref{Cahill-Narayan-Fibonacci-2004-Thm} were applied in the papers~\cite{era-905.tex, ACM-23-0025-S.tex, 2Closed-Bern-Polyn2.tex, mathematics-131192.tex, Slovaca-4738.tex, CDM-68111.tex, Derange-Hess-Det-S.tex, axioms-2464084.tex} and closely related references therein.
\par
In~\cite[Definition~11.2]{Charalambides-book-2002} and~\cite[p.~134, Theorem~A]{Comtet-Combinatorics-74}, the partial Bell polynomials, or say, the Bell polynomials of the second kind, denoted by $\bell_{n,k}(x_1,x_2,\dotsc,x_{n-k+1})$ for $n\ge k\ge0$, are defined by
\begin{equation*}%\label{Bell2nd-Dfn-Eq}
\bell_{n,k}(x_1,x_2,\dotsc,x_{n-k+1})=\sum_{\substack{1\le i\le n-k+1\\ \ell_i\in\{0\}\cup\mathbb{N}\\ \sum_{i=1}^{n-k+1}i\ell_i=n\\
\sum_{i=1}^{n-k+1}\ell_i=k}}\frac{n!}{\prod_{i=1}^{n-k+1}\ell_i!} \prod_{i=1}^{n-k+1}\biggl(\frac{x_i}{i!}\biggr)^{\ell_i}.
\end{equation*}
The famous Fa\`a di Bruno formula, see~\cite[Theorem~11.4]{Charalambides-book-2002} and~\cite[p.~139, Theorem~C]{Comtet-Combinatorics-74}, can be described in terms of $\bell_{n,k}(x_1,x_2,\dotsc,x_{n-k+1})$ by
\begin{equation}\label{Bruno-Bell-Polynomial}
\frac{\operatorname{d}^n[f\circ h(x)]}{\td x^n}=\sum_{k=0}^nf^{(k)}(h(x)) \bell_{n,k}\bigl(h'(x),h''(x),\dotsc,h^{(n-k+1)}(x)\bigr)
\end{equation}
for $n\in\mathbb{N}_0$. The partial Bell polynomials $\bell_{n,k}(x_1,x_2,\dotsc,x_{n-k+1})$ satisfy
\begin{align}\label{Bell-Stir1st=eq}
\bell_{n,k}\biggl(\frac{1!}2,\frac{2!}3,\dotsc,\frac{(n-k+1)!}{n-k+2}\biggr)
=\frac{(-1)^{n-k}}{k!}&\sum_{m=0}^k(-1)^m\binom{k}{m}\frac{s(n+m,m)}{\binom{n+m}{m}},\\
\label{Bell(n-k)}
\bell_{n,k}\bigl(abx_1,ab^2x_2,\dotsc,ab^{n-k+1}x_{n-k+1}\bigr) &=a^kb^n\bell_{n,k}(x_1,x_2,\dotsc,x_{n-k+1}),\\
\label{Bell-stirling}
\bell_{n,k}(\underbrace{1,1,\dotsc,1}_{n-k+1})&=S(n,k),\\
\label{Bell=0!s(n-k)}
\bell_{n,k}(0!,1!,2!,\dotsc,(n-k)!)&=(-1)^{n-k}s(n,k),
\end{align}
and
\begin{equation}\label{B-S-frac-value}
\bell_{n,k}\biggl(\frac12, \frac13,\dotsc,\frac1{n-k+2}\biggr)
=\frac{n!}{(n+k)!}\sum_{\ell=0}^k(-1)^{k-\ell}\binom{n+k}{k-\ell}S(n+\ell,\ell);
\end{equation}
These identities can be found in~\cite[p.~412]{Charalambides-book-2002}, \cite[p.~135]{Comtet-Combinatorics-74}, \cite[Theorem~1.1]{1st-Stirl-No-adjust.tex}, and~\cite[Sections~1.1, 1.2, 1.7, and~1.8]{Bell-value-elem-funct.tex}, respectively.
\par
Theorem~3.1 in~\cite{integer2Bell2real.tex} states that, for any $\ell\in\mathbb{N}_0$, if the series expansion
\begin{equation*}%\label{Z0-k-power-ser-expan}
f^\ell(z)=\sum_{j=0}^{\infty}C_{\ell,j}\frac{z^j}{j!}
\end{equation*}
is valid, then the series expansion
\begin{equation}\label{Z0-alpha-ser-expan}
f^\alpha(z)=\sum_{n=0}^{\infty}\Biggl[\sum_{k=0}^{n} \frac{(-\alpha)_k}{k!} \sum_{q=0}^{k}(-1)^{q}\binom{k}{q} f^{\alpha-q}(0)C_{q,n}\Biggr] \frac{z^n}{n!}
\end{equation}
is also valid for any $\alpha\in\mathbb{R}$.
\par
For $\lambda\in\mathbb{C}$ and $n\in\mathbb{N}_0$, the $n$th falling factorial is defined~\cite[p.~165]{CDM-68111.tex} by
\begin{equation}\label{Fall-Factorial-Dfn-Eq}
\langle\lambda\rangle_n=
\prod_{k=0}^{n-1}(\lambda-k)=
\begin{cases}
\lambda(\lambda-1)\dotsm(\lambda-n+1), & n\ge1;\\
1,& n=0.
\end{cases}
\end{equation}
For $\lambda\in\mathbb{C}$ and $n\in\mathbb{N}_0$, the $n$th rising factorial $(\lambda)_n$, or say, the Pochhammer symbol or shifted factorial, is defined~\cite[p.~167]{CDM-68111.tex} by
\begin{equation*}%\label{rising-factorial-def}
(\lambda)_n=\prod_{\ell=0}^{n-1}(\lambda+\ell)
=
\begin{cases}
\lambda(\lambda+1)\dotsm(\lambda+n-1), & n\ge1;\\
1, & n=0.
\end{cases}
\end{equation*}
The falling and rising factorials have the relations
\begin{equation*}
(-\lambda)_n=(-1)^n\langle \lambda\rangle_n\quad \text{and}\quad \langle-\lambda\rangle_n=(-1)^n(\lambda)_n
\end{equation*}
for $\lambda\in\mathbb{C}$ and $n\in\mathbb{N}_0$.

\section{Motivations and main results}

At the web site \url{https://math.stackexchange.com/q/307274/} (accessed on 15 March 2025), Gottfried Helms asked to confirm the expansions
\begin{equation}\label{Helms-1stP}
\ln(\te^x+1)=\sum_{k=0}^{\infty}\eta(1-k)\frac{x^k}{k!}
\end{equation}
and
\begin{equation}\label{Helms-2ndP}
\ln(\te^x-1)=\ln x-\sum_{k=1}^{\infty}\zeta(1-k)\frac{x^k}{k!}.
\end{equation}
\par
Among other things, the expansion~\eqref{Helms-1stP} was confirmed in~\cite{CMP6528.tex, log-secant-norm-tail.tex} by

\begin{thm}\label{Helms-variant-ser-thm}
For $|x|<\pi$, we have
\begin{equation}\label{Helms-variant-ser}
\ln\frac{\te^x+1}{2}=\frac{x}{2}+\sum_{k=1}^{\infty}\eta(1-2k)\frac{x^{2k}}{(2k)!}.
\end{equation}
\end{thm}

In Section~\ref{exp(x)+1altern-sec}, we will recite an alternative and nicer proof of Theorem~\ref{Helms-variant-ser-thm} at the site \url{https://math.stackexchange.com/a/4940352/} (accessed on 16 March 2025). Moreover, in view of~\eqref{Sitnik-Bourbaki-reform} and~\eqref{Cahill-Narayan-Fibonacci-2004-Thm}, we will deduce three determinantal expressions and two recursive relations of the Bernoulli numbers $B_{2k}$, respectively.
\par
At the web site \url{https://math.stackexchange.com/a/5045900} (accessed on 15 March 2025), Qi confirmed the expansion~\eqref{Helms-2ndP} by

\begin{thm}\label{Konwn-exp-results-thm}
For $|x|<2\pi$, we have
\begin{equation}\label{Konwn-exp-results}
\ln\frac{\te^x-1}{x}=\frac{x}{2}-\sum_{k=1}^{\infty}\zeta(1-2k)\frac{x^{2k}}{(2k)!}.
\end{equation}
\end{thm}

In Section~\ref{Sec-exp(x)+1Bern}, we will recite the nice proof of Theorem~\ref{Konwn-exp-results-thm} at the site \url{https://math.stackexchange.com/a/5045900} (accessed on 15 March 2025), present an alternative series expansion of the function $\ln\frac{\te^x-1}{x}$ in terms of the Stirling numbers of the second kind $S(n,k)$ by virtue of the Fa\`a di Bruno formula~\eqref{Bruno-Bell-Polynomial} and the combinatorial identity~\eqref{B-S-frac-value}, and then derive an closed-form formula of the Bernoulli numbers $B_{2k}$ and a combinatorial identity involving the Stirling numbers of the second kind $S(n,k)$. In addition, in view of~\eqref{Sitnik-Bourbaki-reform} and~\eqref{Cahill-Narayan-Fibonacci-2004-Thm}, we will deduce a determinantal expression and a recursive relation of the Bernoulli numbers $B_{2k}$.
\par
At the web site \url{https://math.stackexchange.com/q/442620} (accessed on 16 March 2025), the series expansion of the function $\sqrt{\ln(1+x)}\,$ at $x=0$ was asked for. In Section~\ref{q-442620-sec}, we will recite a series expansion of the function $\sqrt{\frac{\ln(1+x)}{x}}\,$ and its proof at \url{https://math.stackexchange.com/a/4657078} (accessed on 16 March 2025), establish a series expansion of the function $\bigl[\frac{\ln(1+x)}{x}\bigr]^r$ for $r\in\mathbb{R}$, derive a recursive relation of the Stirling numbers of the first kind $s(n,k)$ and a closed-form formula for the Bernoulli numbers of the second kind $b_n$, which can be generated~\cite{Filomat-36-73-1.tex, Bernoulli2nd-Property.tex, 2rdBern-Det-Zhao.tex} by
\begin{equation}\label{bernoulli-second-dfn}
\frac{x}{\ln(1+x)}=\sum_{n=0}^\infty b_nx^n, \quad |x|<1.
\end{equation}
\par
At the site \url{https://math.stackexchange.com/q/413492} (accessed on 16 March 2025), the series expansion of $\frac{1}{\sqrt{\te^x-1}\,}$ at $x=0$ was asked for. In Section~\ref{sec-exp-Bern-2ndStirl}, we will present a Maclaurin series expansion of the function $\sqrt{\frac{x}{\te^x-1}}\,$ by virtue of the Fa\`a di Bruno formula~\eqref{Bruno-Bell-Polynomial} and the identity~\eqref{B-S-frac-value} in terms of the Stirling numbers of the second kind $S(n,k)$, recite a series expansion of the function $\sqrt{\frac{x}{\te^x-1}}\,$ and its proof at the site \url{https://math.stackexchange.com/a/4657245} (accessed on 15 March 2025) in terms of the Stirling numbers of the first and second kinds, establish a series expansion of the function $\bigl(\frac{\te^x-1}{x}\bigr)^r$ for $r\in\mathbb{R}$, derive three combinatorial identities involving the Stirling numbers of the first and second kinds $s(n,k)$ and $S(n,k)$, deduce closed-form formulas of the Bernoulli numbers $B_{2n}$ and the generalized Bernoulli numbers $B_{n}^{(r)}$ for $r\in\mathbb{R}$.
\par
In Section~\ref{connection-sec}, we will discover two identities connecting the Stirling numbers of the first and second kinds $s(n,k)$ and $S(n,k)$.
\par
We emphasize that the generating functions
\begin{equation*}
\biggl[\frac{\ln(1+x)}{x}\biggr]^r \quad\text{and}\quad \biggl(\frac{\te^x-1}{x}\biggr)^r
\end{equation*}
for $r\in\mathbb{R}$ unify the generating functions of the Stirling numbers of the first and second kinds $s(n,k)$ and $S(n,k)$, the Bernoulli numbers $B_n$, the Bernoulli numbers of the second kind $b_n$, and the generalized Bernoulli numbers $B_n^{(r)}$, respectively. Therefore, the investigation of this paper is of much significance in combinatorial number theory.

\section{An alternative proof of Theorem~\ref{Helms-variant-ser-thm} and three determinantal expressions and recursive relations of Bernoulli numbers}\label{exp(x)+1altern-sec}

In the paper~\cite{CMP6528.tex}, there were three proofs of Theorem~\ref{Helms-variant-ser-thm}. In this section, we recite an alternative and nicer proof of Theorem~\ref{Helms-variant-ser-thm} as follows. Moreover, in view of~\eqref{Sitnik-Bourbaki-reform} and~\eqref{Cahill-Narayan-Fibonacci-2004-Thm}, we deduce three determinantal expressions and recursive relations of the Bernoulli numbers $B_{2k}$, respectively.

\begin{proof}[An alternative proof of Theorem~\ref{Helms-variant-ser-thm}]
This nicer proof is a slightly revised version of the answer at \url{https://math.stackexchange.com/a/4940352/} (accessed on 16 March 2025).
\par
In~\cite[p.~55]{Gradshteyn-Ryzhik-Table-8th}, we find the Maclaurin power series expansion
\begin{equation}\label{log-cosine-series-expansion}
\ln\cos x=-\sum_{k=1}^{\infty}2^{2k}\bigl(2^{2k}-1\bigr)\frac{|B_{2k}|}{2k}\frac{x^{2k}}{(2k)!},\quad |x|<\frac{\pi}2.
\end{equation}
By the expansion~\eqref{log-cosine-series-expansion}, it is not difficult to see that
\begin{align}
\ln\frac{\te^x+1}{2}&=\frac{x}{2}+\ln\frac{\te^{x/2}+\te^{-x/2}}{2}\notag\\
&=\frac{x}{2}+\ln\cosh\frac{x}{2}\notag\\
&=\frac{x}{2}+\ln\cos\frac{x\ti}{2}\notag\\
&=\frac{x}{2}+\sum_{k=1}^{\infty}\bigl(2^{2k}-1\bigr)\frac{B_{2k}}{2k}\frac{x^{2k}}{(2k)!} \label{thm1proof-derivation}\\
&=\frac{x}{2}-\sum_{k=1}^{\infty}\bigl(2^{2k}-1\bigr)\zeta(1-2k)\frac{x^{2k}}{(2k)!}\notag\\
&=\frac{x}{2}+\sum_{k=1}^{\infty}\eta(1-2k)\frac{x^{2k}}{(2k)!}\notag
\end{align}
for $|x|<\pi$, where $\ti=\sqrt{-1}\,$ is the imaginary unit in complex analysis and we used the identity
\begin{equation}\label{abram-23.2.15}
\zeta(1-2k)=-\frac{B_{2k}}{2k}, \quad k\in\mathbb{N},
\end{equation}
found in~\cite[p.~807, Entries~23.2.14 and~23.2.15]{abram}, and the equality~\eqref{eta-zeta-equal}. The series expansion~\eqref{Helms-variant-ser} in Theorem~\ref{Helms-variant-ser-thm} is thus complete.
\end{proof}

\begin{rem}
From the above proof of Theorem~\ref{Helms-variant-ser-thm}, it follows that
\begin{equation}\label{log-cosh-ser}
\ln\cosh x=\sum_{k=1}^{\infty}\bigl(2^{2k}-1\bigr)2^{2k}\frac{B_{2k}}{2k}\frac{x^{2k}}{(2k)!}, \quad |x|<\pi.
\end{equation}
This means that
\begin{equation*}
\bigl(2^{2k}-1\bigr)2^{2k}\frac{B_{2k}}{2k}=\lim_{x\to0}(\ln\cosh x)^{(2k)}=\lim_{x\to0}\biggl(\frac{\sinh x}{\cosh x}\biggr)^{(2k-1)}
\end{equation*}
and $\lim_{x\to0}(\ln\cosh x)^{(2k-1)}=0$ for $k\in\mathbb{N}$. In light of~\eqref{Sitnik-Bourbaki-reform}, we deduce
\begin{equation}\label{Bernoulli-Determin-One}
B_{2k}
=-\frac{k}{\bigl(2^{2k}-1\bigr)2^{2k-1}}
\begin{vmatrix}
0 & 1 & 0 & 0&0&\dotsm & 0 & 0\\
1 & 0 & 1 & 0&0&\dotsm& 0 & 0\\
0 & 1 & 0 & 1&0&\dotsm & 0 &0\\
1 & 0 & \binom{3}{1} & 0&1&\dotsm & 0 &0\\
\vdots & \vdots & \vdots & \vdots & \vdots& \ddots & \vdots & \vdots\\
1 & 0 & \binom{2k-5}{1} & 0&\binom{2k-5}{3}&\dotsm & 0 & 0\\
0 & 1 & 0 &\binom{2k-4}{2} &0&\dotsm & 0 & 0\\
1 & 0 & \binom{2k-3}{1} & 0 &\binom{2k-3}{3}&\dotsm & 1 & 0\\
0 & 1 & 0 & \binom{2k-2}{2}&0&\dotsm & 0 & 1\\
1 & 0 & \binom{2k-1}{1} & 0&\binom{2k-1}{3}&\dotsm & \binom{2k-1}{2k-3} & 0
\end{vmatrix}_{(2k)\times(2k)}
\end{equation}
for $k\in\mathbb{N}_0$.
Furthermore, in view of~\eqref{Cahill-Narayan-Fibonacci-2004-Thm}, we conclude
\begin{equation}\label{Bernou-REcurs-Relat-Eq}
B_{2k}=\frac{k}{\bigl(2^{2k}-1\bigr)2^{2k-1}}\Biggl[1-\sum_{\ell=1}^{k-1} \binom{2k-1}{2\ell-1} \bigl(2^{2\ell}-1\bigr)2^{2\ell} \frac{B_{2\ell}}{2\ell}\Biggr]
\end{equation}
for $k\in\mathbb{N}$.
\end{rem}

\begin{rem}
The series expansion~\eqref{log-cosine-series-expansion} means that
\begin{equation*}
-2^{2k}\bigl(2^{2k}-1\bigr)\frac{|B_{2k}|}{2k}=\lim_{x\to0}(\ln\cos x)^{(2k)}=-\lim_{x\to0}\biggl(\frac{\sin x}{\cos x}\biggr)^{(2k-1)}
\end{equation*}
and $\lim_{x\to0}(\ln\cos x)^{(2k-1)}=0$ for $k\in\mathbb{N}$.
By the derivative formula~\eqref{Sitnik-Bourbaki-reform}, we obtain a determinantal expression
\begin{multline}\label{Bernoulli-Determin-two}
B_{2k}=(-1)^{k+1}\frac{2k}{2^{2k}\bigl(2^{2k}-1\bigr)}\lim_{x\to0}\biggl(\frac{\sin x}{\cos x}\biggr)^{(2k-1)}\\
=\frac{(-1)^{k}k}{2^{2k-1}\bigl(2^{2k}-1\bigr)}
\begin{vmatrix}
0 & 1 & 0 &\dotsm & 0 & 0\\
1 & 0 & 1 &\dotsm& 0 & 0\\
0 & -1 & 0 &\dotsm & 0 &0\\
-1 & 0 & -\binom{3}{1} &\dotsm & 0 &0\\
\vdots & \vdots & \vdots & \ddots & \vdots & \vdots\\
\sin\frac{(2k-3)\pi}{2} & \cos\frac{(2k-3)\pi}{2} & \binom{2k-3}{1}\cos\frac{(2k-4)\pi}{2} & \dotsm & 1 & 0\\
\sin\frac{(2k-2)\pi}{2} & \cos\frac{(2k-2)\pi}{2} & \binom{2k-2}{1}\cos\frac{(2k-3)\pi}{2} & \dotsm & 0 & 1\\
\sin\frac{(2k-1)\pi}{2} & \cos\frac{(2k-1)\pi}{2} & \binom{2k-1}{1}\cos\frac{(2k-2)\pi}{2} & \dotsm & -\binom{2k-1}{2k-3} & 0
\end{vmatrix}_{(2k)\times(2k)}
\end{multline}
for $k\in\mathbb{N}$.
By virtue of the recursive relation~\eqref{Cahill-Narayan-Fibonacci-2004-Thm}, we derive a recursive relation~\eqref{Bernou-REcurs-Relat-Eq} again.
\end{rem}

\begin{rem}
From the derivation in~\eqref{thm1proof-derivation}, it follows that
\begin{equation}\label{exp+1-ser-log}
\ln\frac{\te^x+1}{2}
=\frac{x}{2}+\sum_{k=1}^{\infty}\bigl(2^{2k}-1\bigr)\frac{B_{2k}}{2k}\frac{x^{2k}}{(2k)!}, \quad |x|<\pi.
\end{equation}
This means that
\begin{equation*}
\bigl(2^{2k}-1\bigr)\frac{B_{2k}}{2k}=\lim_{x\to0}\biggl(\ln\frac{\te^x+1}{2}\biggr)^{(2k)}
=\lim_{x\to0}\biggl(\frac{\te^x}{\te^x+1}\biggr)^{(2k-1)}
\end{equation*}
for $k\in\mathbb{N}$. By the derivative formula~\eqref{Sitnik-Bourbaki-reform}, we obtain a determinantal expression
\begin{equation}\label{Bernoulli-Determin-3}
B_{2k}
=-\frac{k}{2^{2k-1}\bigl(2^{2k}-1\bigr)}
\begin{vmatrix}
1 & 2 & 0 &0&\dotsm & 0 & 0\\
1 & 1 & 2 &0&\dotsm& 0 & 0\\
1 & 1 & \binom{2}{1} &2&\dotsm & 0 &0\\
1 & 1 & \binom{3}{1} &\binom{3}{2}&\dotsm & 0 &0\\
\vdots & \vdots & \vdots &  \vdots &\ddots & \vdots & \vdots\\
1 & 1 & \binom{2k-3}{1} & \binom{2k-3}{2}&\dotsm & 2 & 0\\
1 & 1 & \binom{2k-2}{1} & \binom{2k-2}{2}&\dotsm & \binom{2k-2}{2k-3} & 2\\
1 & 1 & \binom{2k-1}{1} & \binom{2k-1}{2}&\dotsm & \binom{2k-1}{2k-3} & \binom{2k-1}{2k-2}
\end{vmatrix}_{(2k)\times(2k)}
\end{equation}
for $k\in\mathbb{N}$. By virtue of the recursive relation~\eqref{Cahill-Narayan-Fibonacci-2004-Thm}, we derive a recursive relation
\begin{equation}\label{Bernoulli-Recursive-2}
B_{2k}=\frac{k}{2(2^{2k}-1)}\Biggl[1-\sum_{j=1}^{k-1}\binom{2k-1}{2j-1} \frac{2^{2j}-1}{j}B_{2j}\Biggr]
\end{equation}
for $k\in\mathbb{N}$.
\end{rem}

\section{Proofs of Theorem~\ref{Konwn-exp-results-thm} and explicit formulas of Bernoulli numbers}\label{Sec-exp(x)+1Bern}

In this section, we recite the proof of Theorem~\ref{Konwn-exp-results-thm} at the site \url{https://math.stackexchange.com/a/5045900} (accessed on 15 March 2025), present an alternative proof of Theorem~\ref{Konwn-exp-results-thm}, establish an alternative Maclaurin expansion of the logarithmic function $\ln\frac{\te^x-1}{x}$ by virtue of the Fa\`a di Bruno formula~\eqref{Bruno-Bell-Polynomial} and the identity~\eqref{B-S-frac-value}, derive an explicit formula, a determinantal expression, and a recursive relation of the Bernoulli numbers $B_{2k}$, and deduce a combinatorial identity for the Stirling numbers of the second kind $S(n,k)$.

\subsection{Two proofs of Theorem~\ref{Konwn-exp-results-thm}}
In this subsection, we provides two nice proofs of Theorem~\ref{Konwn-exp-results-thm}.

\begin{proof}[First proof of Theorem~\ref{Konwn-exp-results-thm}]
This proof is a slightly revised version of the nice answer at \url{https://math.stackexchange.com/a/5045900} (accessed on 15 March 2025).
\par
It is not difficult to see that
\begin{equation*}
\ln\frac{\te^x-1}{x}=\frac{x}{2}+\ln\frac{\sinh(x/2)}{x/2}
=\frac{x}{2}+\ln\frac{\sin(x\ti/2)}{x\ti/2},
\end{equation*}
where $\ti=\sqrt{-1}\,$ is the imaginary unit in complex analysis.
In~\cite[p.~75, Entry~4.3.71]{abram} and~\cite[p.~55]{Gradshteyn-Ryzhik-Table-8th}, we find the Maclaurin expansion
\begin{equation*}
\ln\frac{\sin z}{z}=-\sum_{k=1}^{\infty}\frac{2^{2k-1}}{k}|B_{2k}|\frac{z^{2k}}{(2k)!}, \quad |z|<\pi.
\end{equation*}
Hence, it follows that
\begin{equation*}
\ln\frac{\sin(x\ti/2)}{x\ti/2} =\sum_{k=1}^{\infty}\frac{B_{2k}}{2k}\frac{x^{2k}}{(2k)!}, \quad |x|<2\pi.
\end{equation*}
Accordingly, we derive
\begin{align*}
\ln\frac{\te^x-1}{x}&=\frac{x}{2}+\sum_{k=1}^{\infty}\frac{B_{2k}}{2k}\frac{x^{2k}}{(2k)!}\\
&=\frac{x}{2}-\sum_{k=1}^{\infty}\zeta(1-2k)\frac{x^{2k}}{(2k)!}, \quad |x|<2\pi,\notag
\end{align*}
where we used the identity~\eqref{abram-23.2.15}.
The first proof of the expansion~\eqref{Konwn-exp-results} in Theorem~\ref{Konwn-exp-results-thm} is thus complete.
\end{proof}

\begin{proof}[Second proof of Theorem~\ref{Konwn-exp-results-thm}]
The idea of this proof comes from~\cite[Section~1.1]{log-exp-expan-Sym.tex} and~\cite[Section~3.3]{CMP6528.tex}. This proof was announced at the web site \url{https://math.stackexchange.com/a/5050118/} (accessed on 27 March 2025).
\par
A simple differentiation yields
\begin{align*}
\biggl(\ln\frac{\te^x-1}{x}\biggr)'&=1-\frac{1}{x}+\frac{1}{x}\frac{x}{\te^x-1}\\
&=1-\frac{1}{x}+\frac{1}{x}\Biggl[1-\frac{x}2+\sum_{j=1}^\infty B_{2j}\frac{x^{2j}}{(2j)!}\Biggr]\\
&=\frac{1}{2}+\sum_{j=1}^\infty B_{2j}\frac{x^{2j-1}}{(2j)!}, \quad |x|<2\pi,
\end{align*}
where we used the Maclaurin power series expansion~\eqref{Bernoulli-Gen-Eq}. Integrating over the interval $(0,x)$ on both sides of the above equality at the very ends yields
\begin{equation}\label{F1(x)-ser-expan}
\ln\frac{\te^x-1}{x}=\frac{x}{2}+\sum_{j=1}^\infty\frac{B_{2j}}{2j}\frac{x^{2j}}{(2j)!}, \quad |x|<2\pi.
\end{equation}
Substituting~\eqref{abram-23.2.15} into~\eqref{F1(x)-ser-expan} leads to~\eqref{Konwn-exp-results}.
The second proof of Theorem~\ref{Konwn-exp-results-thm} is thus complete.
\end{proof}

\begin{rem}
From the above proof of Theorem~\ref{Konwn-exp-results-thm}, it follows that
\begin{equation}\label{log-sinh-x-ser}
\ln\frac{\sinh x}{x}
=\sum_{k=1}^{\infty}2^{2k}\frac{B_{2k}}{2k}\frac{x^{2k}}{(2k)!}, \quad |x|<\pi.
\end{equation}
\end{rem}

\begin{rem}
The series expansion~\eqref{F1(x)-ser-expan} means that
\begin{align*}
\frac{B_{2k}}{2k}&=\lim_{x\to0}\biggl(\ln\frac{\te^x-1}{x}\biggr)^{(2k)}\\
&=\lim_{x\to0}\biggl(\ln\int_0^1\te^{xv}\td v\biggr)^{(2k)}\\
&=\lim_{x\to0}\Biggl[\frac{\int_0^1v\te^{xv}\td v}{\int_0^1\te^{xv}\td v}\Biggr]^{(2k-1)}
\end{align*}
and $\lim_{x\to0}\bigl(\ln\frac{\te^x-1}{x}\bigr)^{(2k+1)}=0$ for $k\in\mathbb{N}$. By virtue of the derivative formula~\eqref{Sitnik-Bourbaki-reform}, we acquire the determinantal expression
\begin{equation}\label{Bernoulli-Determin-4}
B_{2k}
=-2k\begin{vmatrix}
\frac{1}{2} & 1 & 0 &\dotsm & 0 & 0\\
\frac{1}{3} & \frac{1}{2} & 1 &\dotsm& 0 & 0\\
\frac{1}{4} & \frac{1}{3} & \frac{1}{2}\binom{2}{1} &\dotsm & 0 &0\\
\vdots & \vdots & \vdots & \ddots & \vdots & \vdots\\
\frac{1}{2k-1} & \frac{1}{2k-2} & \frac{1}{2k-3}\binom{2k-3}{1} & \dotsm & 1 & 0\\
\frac{1}{2k} & \frac{1}{2k-1} & \frac{1}{2k-2}\binom{2k-2}{1} & \dotsm & \frac{1}{2}\binom{2k-2}{2k-3} & 1\\
\frac{1}{2k+1} & \frac{1}{2k} & \frac{1}{2k-1}\binom{2k-1}{1} & \dotsm & \frac{1}{3}\binom{2k-1}{2k-3} & \frac{1}{2}\binom{2k-1}{2k-2}
\end{vmatrix}, \quad k\in\mathbb{N}.
\end{equation}
Making use of the recursive relation~\eqref{Cahill-Narayan-Fibonacci-2004-Thm} leads to the recursive relation
\begin{equation}\label{Bernoulli-Recursive-T3}
B_{2k}=2k\Biggl[\frac{2k-1}{4k(2k+1)}-\sum_{\ell=1}^{k-1}\frac{1}{2k-2\ell+1} \binom{2k-1}{2\ell-1} \frac{B_{2\ell}}{2\ell}\Biggr], \quad k\in\mathbb{N}.
\end{equation}
\end{rem}

\subsection{An explicit and closed-form formula of Bernoulli numbers}

Theorem~4 in~\cite{CMP6528.tex} reads that
\begin{equation}\label{Helms-2nd-Stirling-Ser}
\ln\frac{\te^x+1}{2}=\sum_{k=1}^\infty \Biggl[\sum_{j=1}^k(-1)^{j-1}\frac{(j-1)!}{2^{j}}S(k,j)\Biggr]\frac{x^k}{k!},\quad |x|<\pi.
\end{equation}
Comparing~\eqref{Helms-2nd-Stirling-Ser} with~\eqref{Helms-variant-ser} and employing~\eqref{eta-zeta-equal} and~\eqref{abram-23.2.15} yield~\cite[Remark~2]{CMP6528.tex} the closed-form formula
\begin{equation}\label{Helms-variant-ser27}
B_{2k}=\frac{2k}{2^{2k}-1} \sum_{j=1}^{2k}(-1)^{j-1}\frac{(j-1)!}{2^{j}}S(2k,j), \quad k\in\mathbb{N}
\end{equation}
and the combinatorial identity
\begin{equation}\label{Helms-variant=0}
\sum_{j=1}^{2k+1}\frac{(-1)^{j}}{2^{j}}(j-1)!S(2k+1,j)=0, \quad k\in\mathbb{N}.
\end{equation}
\par
In the papers~\cite{v2Bernoulli-ID-Stack.tex, RMJ-7006.tex, Genocchi-Stirling.tex, ANLY-D-12-1238.tex, Guo-Qi-JANT-Bernoulli.tex, exp-derivative-sum-Combined.tex, 2Closed-Bern-Polyn2.tex, Bernoulli-Stirling-4P.tex, mathematics-131192.tex} and at the sites
\url{https://math.stackexchange.com/a/4256913/} (accessed on 26 March 2025) and
\url{https://math.stackexchange.com/a/4256915/} (accessed on 26 March 2025),
many closed-form formulas of the Bernoulli numbers and polynomials $B_n$ and $B_n(t)$ in terms of central factorial numbers of the second kind, the Stirling numbers of the second kind, and determinants were collected, reviewed, surveyed, rediscovered, and established.
\par
Motivated by the above ideas and approaches of deriving~\eqref{Helms-2nd-Stirling-Ser}, \eqref{Helms-variant-ser27}, and~\eqref{Helms-variant=0}, we present the following theorem.

\begin{thm}\label{Bell-exp-results-thm}
For $|x|<2\pi$, we have
\begin{equation}\label{Bell-exp-results}
\ln\frac{\te^x-1}{x}=-\sum_{k=1}^{\infty}\Biggl[\frac{1}{\binom{2k}{k}} \sum_{\ell=1}^k\frac{(-1)^{\ell}}{\ell}\binom{2k}{k+\ell}S(k+\ell,\ell)\Biggr]\frac{x^k}{k!}.
\end{equation}
For $k\in\mathbb{N}$, we have
\begin{equation}\label{Equal=0-Stirl-Bern}
B_{2k}=\frac{2k}{\binom{4k}{2k}} \sum_{\ell=1}^{2k}\frac{(-1)^{\ell}}{\ell}\binom{4k}{2k+\ell}S(2k+\ell,\ell)
\end{equation}
and
\begin{equation}\label{Equal=0-Stirl}
\sum_{\ell=1}^{2k+1}\frac{(-1)^{\ell}}{\ell}\binom{4k+2}{2k+\ell+1}S(2k+\ell+1,\ell)=0.
\end{equation}
\end{thm}

\begin{proof}
It is not difficult to verify that
\begin{equation*}
\ln\frac{\te^x-1}{x}=\ln \int_0^1\te^{xv}\td v\triangleq \ln u(x), \quad x\in\mathbb{R}.
\end{equation*}
Then, by virtue of the Fa\`a di Bruno formula~\eqref{Bruno-Bell-Polynomial}, we obtain
\begin{align*}
 \biggl(\ln\frac{\te^x-1}{x}\biggr)^{(k)}
 &=\sum_{j=0}^{k}(\ln u)^{(j)}\bell_{k,j}\bigl(u'(x), u''(x),\dotsc, u^{(k-j+1)}(x)\bigr)\\
 &=\sum_{j=0}^{k}\frac{(-1)^{j-1}(j-1)!}{u^j} \bell_{k,j} \Biggl(\biggl(\int_0^1\te^{xv}\td v\biggr)', \biggl(\int_0^1\te^{xv}\td v\biggr)'',\\
 &\quad \dotsc, \biggl(\int_0^1\te^{xv}\td v\biggr)^{(k-j+1)}\Biggr)\\
 &=\sum_{j=0}^{k}\frac{(-1)^{j-1}(j-1)!}{[u(x)]^j} \bell_{k,j} \biggl(\int_0^1v\te^{xv}\td v,\\
 &\quad \int_0^1v^2\te^{xv}\td v,\dotsc, \int_0^1v^{k-j+1}\te^{xv}\td v\biggr)\\
 &\to\sum_{j=0}^{k} (-1)^{j-1}(j-1)! \bell_{k,j} \biggl(\int_0^1v\td v, \int_0^1v^2\td v,\\
 &\quad \dotsc, \int_0^1v^{k-j+1}\td v\biggr), \quad x\to0\\
 &=\sum_{j=0}^{k} (-1)^{j-1}(j-1)! \bell_{k,j} \biggl(\int_0^1v\td v, \int_0^1v^2\td v,\dotsc, \int_0^1v^{k-j+1}\td v\biggr)\\
 &=\sum_{j=0}^{k} (-1)^{j-1}(j-1)! \bell_{k,j} \biggl(\frac{1}{2}, \frac{1}{3},\dotsc, \frac1{k-j+2}\biggr)\\
 &=\sum_{j=0}^{k} (-1)^{j-1} \frac{(j-1)!k!}{(k+j)!}\sum_{\ell=0}^j(-1)^{j-\ell}\binom{k+j}{j-\ell}S(k+\ell,\ell)\\
 &=-\sum_{j=0}^{k} \frac{1}{j\binom{k+j}{j}}\sum_{\ell=0}^j(-1)^{\ell}\binom{k+j}{j-\ell}S(k+\ell,\ell)\\
 &=-\sum_{j=0}^{k} \frac{1}{j\binom{k+j}{j}}\binom{k+j}{j}S(k,0) \\
 &\quad -\sum_{j=0}^{k} \frac{1}{j\binom{k+j}{j}} \sum_{\ell=1}^j(-1)^{\ell}\binom{k+j}{j-\ell}S(k+\ell,\ell)\\
 &=-\sum_{\ell=1}^k(-1)^{\ell}\Biggl[\sum_{j=\ell}^{k} \frac{1}{j\binom{k+j}{j}}\binom{k+j}{j-\ell}\Biggr] S(k+\ell,\ell)\\
 &=-\sum_{\ell=1}^k(-1)^{\ell}\frac{k!}{(k+\ell)!} \Biggl[\sum_{j=\ell}^{k} \frac{(j-1)!}{(j-\ell)!}\Biggr] S(k+\ell,\ell)\\
 &=-\sum_{\ell=1}^k(-1)^{\ell}\frac{k!}{(k+\ell)!} \frac{k!}{\ell (k-\ell)!} S(k+\ell,\ell)\\
 &=-(k!)^2\sum_{\ell=1}^k\frac{(-1)^{\ell}}{\ell}\frac{1}{(k+\ell)!} \frac{1}{(k-\ell)!} S(k+\ell,\ell)
\end{align*}
for $k\in\mathbb{N}$, where we used the identity~\eqref{B-S-frac-value} and the identity
\begin{equation*}
\sum_{j=\ell}^{k} \frac{(j-1)!}{(j-\ell)!}
=(\ell-1)!\sum_{j=\ell}^{k} \binom{j-1}{\ell-1}
=(\ell-1)!\binom{k}{\ell}
=\frac{k!}{\ell (k-\ell)!}
\end{equation*}
for $\ell\in\mathbb{N}$, deduced by virtue of the identity
\begin{equation}\label{Identity58Spivey-art-2019}
\sum_{k=0}^n\binom{k}{m}=\sum_{k=m}^n\binom{k}{m}=\binom{n+1}{m+1}, \quad n\ge m\in\mathbb{N}_0
\end{equation}
which can be found as Identity~58 in~\cite{Spivey-art-2019}.
Therefore, we conclude the Maclaurin series expansion~\eqref{Bell-exp-results}.
\par
Comparing between~\eqref{Bell-exp-results} and~\eqref{Konwn-exp-results} and considering the relation~\eqref{abram-23.2.15} result in~\eqref{Equal=0-Stirl-Bern} and~\eqref{Equal=0-Stirl} straightforwardly.
The proof of Theorem~\ref{Bell-exp-results-thm} is complete.
\end{proof}

\section{On generating functions of Stirling numbers of first kind}\label{q-442620-sec}

As the most perfect answer to the question at \url{https://math.stackexchange.com/q/442620} (accessed on 16 March 2025), the following theorem was established.

\begin{thm}\label{Sqrt-log-thm}
For $|x|<1$, we have the series expansion
\begin{equation}\label{Sqrt-log-Eq}
\sqrt{\frac{\ln(1+x)}{x}}\,=-\sum_{n=0}^\infty\Biggl[\sum_{k=0}^n\frac{(2k-3)!!}{(2k)!!} \sum_{m=0}^k(-1)^m\binom{k}{m}\frac{s(n+m,m)}{\binom{n+m}{m}}\Biggr]\frac{x^n}{n!}.
\end{equation}
\end{thm}

\begin{proof}
This proof is a slightly revised version of the answer at \url{https://math.stackexchange.com/a/4657078} (accessed on 16 March 2025).
\par
It is well known that
\begin{equation*}
\frac{\ln(1+x)}{x}=\sum_{k=0}^\infty(-1)^k\frac{x^k}{k+1}, \quad |x|<1.
\end{equation*}
This means that
\begin{equation*}
\lim_{x\to0}\biggl[\frac{\ln(1+x)}{x}\biggr]^{(k)}
=(-1)^k\frac{k!}{k+1}, \quad k\ge0.
\end{equation*}
By virtue of the Fa\`a di Bruno formula~\eqref{Bruno-Bell-Polynomial}, we obtain
\begin{align*}
\Biggl[&\sqrt{\frac{\ln(1+x)}{x}}\,\Biggr]^{(n)}
=\sum_{k=0}^n\biggl\langle\frac12\biggr\rangle_k \biggl[\frac{\ln(1+x)}{x}\biggr]^{1/2-k} \\
&\quad\times \bell_{n,k}\Biggl(\biggl[\frac{\ln(1+x)}{x}\biggr]', \biggl[\frac{\ln(1+x)}{x}\biggr]'', \dotsc, \biggl[\frac{\ln(1+x)}{x}\biggr]^{(n-k+1)}\Biggr)\\
&\to\sum_{k=0}^n\biggl\langle\frac12\biggr\rangle_k \bell_{n,k}\biggl(-\frac{1!}{2}, \frac{2!}{3}, \dotsc, (-1)^{n-k+1}\frac{(n-k+1)!}{n-k+2}\biggr), \quad x\to0\\
&=(-1)^n\sum_{k=0}^n(-1)^{k-1}\frac{(2k-3)!!}{2^k} \bell_{n,k}\biggl(\frac{1!}{2}, \frac{2!}{3}, \dotsc, \frac{(n-k+1)!}{n-k+2}\biggr)\\
&=(-1)^n\sum_{k=0}^n(-1)^{k-1}\frac{(2k-3)!!}{2^k} \frac{(-1)^{n-k}}{k!}\sum_{m=0}^k(-1)^m\binom{k}{m}\frac{s(n+m,m)}{\binom{n+m}{m}}\\
&=-\sum_{k=0}^n\frac{(2k-3)!!}{(2k)!!} \sum_{m=0}^k(-1)^m\binom{k}{m}\frac{s(n+m,m)}{\binom{n+m}{m}},
\end{align*}
where we used the formulas~\eqref{Bell-Stir1st=eq} and~\eqref{Bell(n-k)}.
Consequently, we arrive at the Maclaurin series expansion~\eqref{Sqrt-log-Eq}.
The proof of the series expansion~\eqref{Sqrt-log-Eq} in Theorem~\ref{Sqrt-log-thm} is thus complete.
\end{proof}

As a generalization of Theorem~\ref{Sqrt-log-thm} and the equation~\eqref{Stirl-No-First-GF} for generating the Stirling numbers of the first kind $s(n,k)$, the following result is valid.

\begin{thm}\label{real-power-log-thm}
For $|x|<1$ and $r\in\mathbb{R}$, we have the series expansion
\begin{equation}\label{real-power-log-Eq}
\biggl[\frac{\ln(1+x)}{x}\biggr]^r
=\sum_{n=0}^{\infty}\Biggl[\sum_{k=0}^n\frac{(-r)_k}{k!} \sum_{m=0}^k(-1)^m\binom{k}{m}\frac{s(n+m,m)}{\binom{n+m}{m}}\Biggr]\frac{x^n}{n!}.
\end{equation}
\end{thm}

\begin{proof}
As done in the proof of Theorem~\ref{Sqrt-log-thm}, using the Fa\`a di Bruno formula~\eqref{Bruno-Bell-Polynomial} and employing the formulas~\eqref{Bell-Stir1st=eq} and~\eqref{Bell(n-k)}, we acquire
\begin{align*}
\biggl(&\biggl[\frac{\ln(1+x)}{x}\biggr]^r\biggr)^{(n)}
=\sum_{k=0}^n\langle r\rangle_k \biggl[\frac{\ln(1+x)}{x}\biggr]^{r-k} \\
&\quad\times \bell_{n,k}\Biggl(\biggl[\frac{\ln(1+x)}{x}\biggr]', \biggl[\frac{\ln(1+x)}{x}\biggr]'', \dotsc, \biggl[\frac{\ln(1+x)}{x}\biggr]^{(n-k+1)}\Biggr)\\
&\to\sum_{k=0}^n\langle r\rangle_k \bell_{n,k}\biggl(-\frac{1!}{2}, \frac{2!}{3}, \dotsc, (-1)^{n-k+1}\frac{(n-k+1)!}{n-k+2}\biggr), \quad x\to0\\
&=(-1)^n\sum_{k=0}^n\langle r\rangle_k \bell_{n,k}\biggl(\frac{1!}{2}, \frac{2!}{3}, \dotsc, \frac{(n-k+1)!}{n-k+2}\biggr)\\
&=(-1)^n\sum_{k=0}^n\langle r\rangle_k \frac{(-1)^{n-k}}{k!}\sum_{m=0}^k(-1)^m\binom{k}{m}\frac{s(n+m,m)}{\binom{n+m}{m}}\\
&=\sum_{k=0}^n(-1)^{k} \frac{\langle r\rangle_k}{k!}\sum_{m=0}^k(-1)^m\binom{k}{m}\frac{s(n+m,m)}{\binom{n+m}{m}}\\
&=\sum_{k=0}^n \frac{(-r)_k}{k!}\sum_{m=0}^k(-1)^m\binom{k}{m}\frac{s(n+m,m)}{\binom{n+m}{m}},\quad n\in\mathbb{N}_0.
\end{align*}
Consequently, the series expansion~\eqref{real-power-log-Eq} in Theorem~\ref{real-power-log-thm} is complete.
\end{proof}

\begin{rem}
Comparing~\eqref{real-power-log-Eq} with~\eqref{Stirl-No-First-GF} results in the combinatorial identity
\begin{equation}\label{diag-1st-stirl-eq}
s(n+r,r)
=\binom{n+r}{r}\sum_{k=0}^n\frac{(-r)_k}{k!} \sum_{m=0}^k(-1)^m\binom{k}{m}\frac{s(n+m,m)}{\binom{n+m}{m}}, \quad n,r\in\mathbb{N}_0.
\end{equation}
This identity is seemingly a generalization of the diagonal recurrence relations discussed in~\cite{1st-Stirl-No-adjust.tex, AADM-2821.tex} for the Stirling numbers of the first kind $s(n,k)$.
\end{rem}

\begin{rem}
The Bernoulli numbers of the second kind $b_n$ can be generated by the equation~\eqref{bernoulli-second-dfn}.
For more information and recent results of the Bernoulli numbers of the second kind $b_n$, please refer to the papers~\cite{Bernoulli-Stirling-Beograd.tex, Filomat-36-73-1.tex, Pyo-Kim-Rim-JNSA.tex, DJK2S-JAA.tex, Bernoulli2nd-Property.tex, 2rdBern-Det-Zhao.tex} and closely related references therein.
\par
Taking $r=-1$ in~\eqref{real-power-log-Eq} and comparing with~\eqref{bernoulli-second-dfn} yields
\begin{align}
b_n&=\frac{1}{n!}\sum_{k=0}^n\sum_{m=0}^k(-1)^m\binom{k}{m}\frac{s(n+m,m)}{\binom{n+m}{m}}\notag\\
&=\frac{1}{n!}\sum_{m=0}^n(-1)^m\Biggl[\sum_{k=m}^n\binom{k}{m}\Biggr]\frac{s(n+m,m)}{\binom{n+m}{m}}\notag\\
&=\frac{1}{n!}\sum_{m=0}^n(-1)^m \binom{n+1}{m+1}\frac{s(n+m,m)}{\binom{n+m}{m}} \label{real-power-log-r=-1}
\end{align}
for $n\in\mathbb{N}_0$,
where we used the combinatorial identity~\eqref{Identity58Spivey-art-2019}.
\end{rem}

\begin{rem}
The Maclaurin series expansion~\eqref{real-power-log-Eq} can also be derived from applying the series expansion~\eqref{Z0-alpha-ser-expan} to $f(z)=\frac{\ln(1+z)}{z}$ and considering the series expansion~\eqref{Stirl-No-First-GF}.
\end{rem}

\section{On generating functions of Stirling numbers of second kind}\label{sec-exp-Bern-2ndStirl}

In this section, we first supply two nice answers to the question at \url{https://math.stackexchange.com/q/413492} (accessed on 16 March 2025).

\begin{thm}\label{Bern-Exp-Ser-2nd-thm}
For $|x|<\pi$, we have the Maclaurin series expansion
\begin{equation}\label{Bern-Exp-Ser-2nd-Eq}
\sqrt{\frac{x}{\te^x-1}}\,
=\sum_{n=0}^{\infty}\Biggl[\sum_{k=0}^{n}\frac{1}{2^{2k}} \frac{\binom{2k}{k}}{\binom{n+k}{k}} \sum_{\ell=0}^k(-1)^{\ell}\binom{n+k}{k-\ell}S(n+\ell,\ell)\Biggr]\frac{x^n}{n!}.
\end{equation}
\end{thm}

\begin{proof}
Denoting
$$
u=u(x)=\frac{\te^x-1}{x}=\int_0^1\te^{xv}\td v,\quad x\in\mathbb{R},
$$
utilizing the Fa\`a di Bruno formula~\eqref{Bruno-Bell-Polynomial}, and applying the identity~\eqref{B-S-frac-value}, we acquire
\begin{align*}
\biggl(\sqrt{\frac{x}{\te^x-1}}\,\biggr)^{(n)}
&=\sum_{k=0}^{n}\bigl(u^{-1/2}\bigr)^{(k)} \bell_{n,k}\bigl(u'(x), u''(x), \dotsc, u^{(n-k+1)}(x)\bigr)\\
&=\sum_{k=0}^{n}\biggl\langle-\frac{1}{2}\biggr\rangle_{k}\frac{1}{u^{1/2+k}} \bell_{n,k}\biggl(\int_0^1v\te^{xv}\td v,\\
&\quad\int_0^1v^2\te^{xv}\td v,\dotsc, \int_0^1v^{n-k+1}\te^{xv}\td v\biggr)\\
&\to\sum_{k=0}^{n}\biggl\langle-\frac{1}{2}\biggr\rangle_{k} \bell_{n,k}\biggl(\int_0^1v\td v, \int_0^1v^2\td v, \\
&\quad\dotsc, \int_0^1v^{n-k+1}\td v\biggr), \quad x\to0\\
&=\sum_{k=0}^{n}\biggl\langle-\frac{1}{2}\biggr\rangle_{k} \bell_{n,k}\biggl(\frac{1}{2}, \frac{1}{3},\dotsc, \frac1{n-k+2}\biggr)\\
&=\sum_{k=0}^{n}(-1)^k\frac{(2k-1)!!}{2^k} \frac{n!}{(n+k)!}\sum_{\ell=0}^k(-1)^{k-\ell}\binom{n+k}{k-\ell}S(n+\ell,\ell)\\
&=\sum_{k=0}^{n}\frac{1}{2^{2k}} \frac{\binom{2k}{k}}{\binom{n+k}{k}}\sum_{\ell=0}^k(-1)^{\ell}\binom{n+k}{k-\ell}S(n+\ell,\ell).
\end{align*}
The proof of Theorem~\ref{Bern-Exp-Ser-2nd-thm} is complete.
\end{proof}

\begin{thm}\label{Bern-Exp-Ser-thm}
For $|x|<\pi$, we have the Maclaurin series expansion
\begin{equation}\label{Bern-Exp-Ser-Eq}
\sqrt{\frac{x}{\te^x-1}}\,
=-\sum_{n=0}^\infty\Biggl[\sum_{k=0}^nS(n,k)\sum_{\ell=0}^k\frac{(2\ell-3)!!}{(2\ell)!!} \sum_{m=0}^\ell(-1)^m\binom{\ell}{m}\frac{s(k+m,m)}{\binom{k+m}{m}}\Biggr]\frac{x^n}{n!}.
\end{equation}
\end{thm}

\begin{proof}
This interesting proof is a slightly revised version of the answer at the site \url{https://math.stackexchange.com/a/4657245} (accessed on 15 March 2025).
\par
Making use of the expansion~\eqref{Sqrt-log-Eq} in Theorem~\ref{Sqrt-log-thm}, we derive
\begin{align*}
\biggl(\sqrt{\frac{x}{\te^x-1}}\,\biggr)^{(n)}
&=\Biggl[\sqrt{\frac{\ln(1+u(x))}{u(x)}}\,\Biggr]^{(n)}, \quad u=u(x)=\te^x-1\\
&=\sum_{k=0}^n\Biggl[\sqrt{\frac{\ln(1+u)}{u}}\,\Biggr]^{(k)} \bell_{n,k}(\te^x,\te^x,\dotsc,\te^x)\\
&=\sum_{k=0}^n\Biggl[\sqrt{\frac{\ln(1+u)}{u}}\,\Biggr]^{(k)} \te^{kx} \bell_{n,k}(1,1,\dotsc,1)\\
&\to \sum_{k=0}^n\lim_{u\to0}\Biggl[\sqrt{\frac{\ln(1+u)}{u}}\,\Biggr]^{(k)} S(n,k), \quad x\to0\\
&=-\sum_{k=0}^n\Biggl[\sum_{\ell=0}^k\frac{(2\ell-3)!!}{(2\ell)!!} \sum_{m=0}^\ell(-1)^m\binom{\ell}{m}\frac{s(k+m,m)}{\binom{k+m}{m}}\Biggr]S(n,k),
\end{align*}
where we used the identities~\eqref{Bell(n-k)} and~\eqref{Bell-stirling}.
Consequently, we arrive at the Maclaurin expansion~\eqref{Bern-Exp-Ser-Eq}. The proof of Theorem~\ref{Bern-Exp-Ser-thm} is thus complete.
\end{proof}

\begin{rem}
Comparing~\eqref{Bern-Exp-Ser-Eq} with~\eqref{Bern-Exp-Ser-2nd-Eq} gives the combinatorial identity
\begin{multline}\label{Conn-Stirl-1st-2nd}
\sum_{k=0}^{n}\frac{1}{2^{2k}} \frac{\binom{2k}{k}}{\binom{n+k}{k}} \sum_{\ell=0}^k(-1)^{\ell}\binom{n+k}{k-\ell}S(n+\ell,\ell)\\
=\sum_{k=0}^nS(n,k)\sum_{\ell=0}^k\frac{(2\ell-3)!!}{(2\ell)!!} \sum_{m=0}^\ell(-1)^m\binom{\ell}{m}\frac{s(k+m,m)}{\binom{k+m}{m}}, \quad n\in\mathbb{N}_0.
\end{multline}
\end{rem}

\begin{thm}\label{Bern-Exp-Ser-Gen-thm}
For $r\in\mathbb{R}$ and $|x|<2\pi$, we have the Maclaurin series expansion
\begin{equation}\label{Bern-Exp-Ser-Gen-Eq}
\biggl(\frac{\te^x-1}{x}\biggr)^r
=\sum_{n=0}^{\infty}\Biggl[\sum_{k=0}^{n} \frac{(-r)_{k}}{(n+k)!} \sum_{\ell=0}^k(-1)^{\ell}\binom{n+k}{k-\ell}S(n+\ell,\ell)\Biggr]x^n.
\end{equation}
\end{thm}

\begin{proof}
As done in the proof of Theorem~\ref{Bern-Exp-Ser-2nd-thm}, utilizing the Fa\`a di Bruno formula~\eqref{Bruno-Bell-Polynomial}, and applying the identity~\eqref{B-S-frac-value}, we arrive at
\begin{align*}
\biggl[\biggl(\frac{\te^x-1}{x}\biggr)^r\biggr]^{(n)}
&=\sum_{k=0}^{n}\bigl(u^r\bigr)^{(k)} \bell_{n,k}\bigl(u'(x), u''(x), \dotsc, u^{(n-k+1)}(x)\bigr)\\
&=\sum_{k=0}^{n}\langle r\rangle_{k}u^{r-k}(x) \bell_{n,k}\biggl(\int_0^1v\te^{xv}\td v,\\
&\quad\int_0^1v^2\te^{xv}\td v,\dotsc, \int_0^1v^{n-k+1}\te^{xv}\td v\biggr)\\
&\to\sum_{k=0}^{n}\langle r\rangle_{k} \bell_{n,k}\biggl(\int_0^1v\td v, \int_0^1v^2\td v, \dotsc, \int_0^1v^{n-k+1}\td v\biggr), \quad x\to0\\
&=\sum_{k=0}^{n}\langle r\rangle_{k} \bell_{n,k}\biggl(\frac{1}{2}, \frac{1}{3},\dotsc, \frac1{n-k+2}\biggr)\\
&=\sum_{k=0}^{n}\langle r\rangle_{k} \frac{n!}{(n+k)!}\sum_{\ell=0}^k(-1)^{k-\ell}\binom{n+k}{k-\ell}S(n+\ell,\ell)\\
&=n!\sum_{k=0}^{n} \frac{(-r)_{k}}{(n+k)!}\sum_{\ell=0}^k(-1)^{\ell}\binom{n+k}{k-\ell}S(n+\ell,\ell).
\end{align*}
The proof of Theorem~\ref{Bern-Exp-Ser-Gen-thm} is thus complete.
\end{proof}

\begin{rem}
Comparing~\eqref{Bern-Exp-Ser-Gen-Eq} for $r=-1$ with~\eqref{Bernoulli-Gen-Eq} leads to
\begin{equation}\label{Equal-Stirl-Bern2nd}
B_{2n}=\sum_{k=0}^{2n}\frac{k!}{(2n+k)!}\sum_{\ell=0}^k(-1)^{\ell}\binom{2n+k}{k-\ell}S(2n+\ell,\ell), \quad n\in\mathbb{N}
\end{equation}
and
\begin{equation}\label{Equal-Stirl-Bern2nd=0}
\sum_{k=0}^{2n+1} \frac{k!}{(2n+k+1)!}\sum_{\ell=0}^k(-1)^{\ell}\binom{2n+k+1}{k-\ell}S(2n+\ell+1,\ell)=0, \quad n\in\mathbb{N}.
\end{equation}
\end{rem}

\begin{rem}
Comparing~\eqref{Bern-Exp-Ser-Gen-Eq} for $r\in\mathbb{N}_0$ with~\eqref{2Stirl-funct-rew} gives the combinatorial identity
\begin{equation}\label{diag-2nd-stirl-eq}
S(n+r,r)
=\frac{(n+r)!}{r!}\sum_{k=0}^{n} \frac{(-r)_{k}}{(n+k)!}\sum_{\ell=0}^k(-1)^{\ell}\binom{n+k}{k-\ell}S(n+\ell,\ell), \quad n,r\in\mathbb{N}_0.
\end{equation}
This identity is seemingly a generalization of the diagonal recurrence relations, discussed in~\cite{MIA-4666.tex}, for the Stirling numbers of the second kind $S(n,k)$.
\end{rem}

\begin{rem}
The expansion~\eqref{Bern-Exp-Ser-Gen-Eq} in Theorem~\ref{Bern-Exp-Ser-Gen-thm} essentially gives a closed-form formula of the generalized Bernoulli numbers $B_n^{(r)}$ in terms of the Stirling numbers of the second kind $S(n,k)$ by
\begin{equation}\label{Equal-Stirl-Bern3rd}
B_n^{(r)}=n!\sum_{k=0}^{n} \frac{(r)_{k}}{(n+k)!}\sum_{\ell=0}^k(-1)^{\ell}\binom{n+k}{k-\ell}S(n+\ell,\ell)
\end{equation}
for $n\in\mathbb{N}_0$ and $r\in\mathbb{R}$. For detailed information on the generalized Bernoulli numbers $B_n^{(r)}$, please refer to~\cite[Chapter~1]{Temme-96-book}.
\end{rem}

\begin{rem}
The Maclaurin series expansion~\eqref{Bern-Exp-Ser-Gen-Eq} can also be concluded from applying~\eqref{Z0-alpha-ser-expan} together with~\eqref{2Stirl-funct-rew}.
\end{rem}

\section{Identities connecting Stirling numbers of first and second kinds}\label{connection-sec}

In this section, motivated by Theorem~\ref{Bern-Exp-Ser-thm} and its proof, we discover the following couple of identities, similar to~\eqref{Conn-Stirl-1st-2nd}, connecting the Stirling numbers of the first and second kinds $s(n,k)$ and $S(n,k)$.

\begin{thm}\label{exp-log-trans-thm}
For $|x|<\pi$ and $r\in\mathbb{R}$, the series expansion
\begin{equation}\label{exp-log-trans-ser}
\biggl(\frac{\te^x-1}{x}\biggr)^r
=\sum_{n=0}^{\infty}\Biggl[\sum_{\ell=0}^{n}S(n,\ell)\sum_{k=0}^\ell\frac{(r)_k}{k!} \sum_{m=0}^k(-1)^m\binom{k}{m}\frac{s(\ell+m,m)}{\binom{\ell+m}{m}}\Biggr]\frac{x^n}{n!}
\end{equation}
is valid. Consequently, the combinatorial identity
\begin{multline}\label{Stirl-1st-2nd-conn}
\sum_{k=0}^{n} \frac{(r)_{k}}{(n+k)!} \sum_{\ell=0}^k(-1)^{\ell}\binom{n+k}{k-\ell}S(n+\ell,\ell)\\
=\frac{1}{n!}\sum_{\ell=0}^{n}S(n,\ell)\sum_{k=0}^\ell\frac{(-r)_k}{k!} \sum_{m=0}^k(-1)^m\binom{k}{m}\frac{s(\ell+m,m)}{\binom{\ell+m}{m}}
\end{multline}
holds for $n\in\mathbb{N}_0$.
\end{thm}

\begin{proof}
Since
\begin{equation}\label{int-repres-exp-log}
\frac{\te^x-1}{x}=\int_0^1\te^{xv}\td v=\frac{u(x)}{\ln[1+u(x)]}=\int_0^1[1+u(x)]^{v}\td v,
\end{equation}
where $u=u(x)=\te^x-1$, we arrive at
\begin{gather*}
\biggl[\biggl(\frac{\te^x-1}{x}\biggr)^r\biggr]^{(n)}
=\frac{\operatorname{d}^n}{\td x^n}\biggl[\biggl(\frac{u(x)}{\ln[1+u(x)]}\biggr)^r\biggr]\\
=\sum_{\ell=0}^{n}\frac{\operatorname{d}^\ell}{\td u^\ell}\biggl[\biggl(\frac{u}{\ln(1+u)}\biggr)^r\biggr] \bell_{n,\ell}\bigl(u'(x), u''(x), \dotsc,u^{n-\ell+1}(x)\bigr)\\
\to\sum_{\ell=0}^{n}\Biggl[\sum_{k=0}^\ell\frac{(r)_k}{k!} \sum_{m=0}^k(-1)^m\binom{k}{m}\frac{s(\ell+m,m)}{\binom{\ell+m}{m}}\Biggr] \bell_{n,\ell}(\underbrace{1,1,\dotsc,1}_{n-\ell+1}), \quad x\to0\\
=\sum_{\ell=0}^{n}\Biggl[\sum_{k=0}^\ell\frac{(r)_k}{k!} \sum_{m=0}^k(-1)^m\binom{k}{m}\frac{s(\ell+m,m)}{\binom{\ell+m}{m}}\Biggr]S(n,\ell),
\end{gather*}
where we used the series expansion~\eqref{real-power-log-Eq} in Theorem~\ref{real-power-log-thm} and the identity~\eqref{Bell-stirling}. Hence, the Maclaurin series expansion~\eqref{exp-log-trans-ser} is thus proved.
\par
Comparing~\eqref{exp-log-trans-ser} with~\eqref{Bern-Exp-Ser-Gen-Eq} gives the combinatorial identity~\eqref{Stirl-1st-2nd-conn}. The proof of Theorem~\ref{exp-log-trans-thm} is complete.
\end{proof}

\begin{rem}
From the definition~\eqref{bernoulli-second-dfn} and the integral representation
\begin{equation*}
\frac{x}{\ln(1+x)}=\int_0^1(1+x)^{v}\td v,\quad |x|<1
\end{equation*}
in~\eqref{int-repres-exp-log}, we recover the integral representation
\begin{align}
b_n&=\frac{1}{n!}\lim_{x\to0}\biggl[\frac{x}{\ln(1+x)}\biggr]^{(n)}\notag\\
&=\frac{1}{n!}\lim_{x\to0}\biggl[\int_0^1(1+x)^{v}\td v\biggr]^{(n)}\notag\\
&=\frac{1}{n!}\lim_{x\to0}\biggl[\int_0^1\langle v\rangle_n(1+x)^{v}\td v\biggr]\notag\\
&=\frac{1}{n!}\int_0^1\langle v\rangle_n\td v\label{falling-2ndbER}
\end{align}
of the Bernoulli numbers of the second kind $b_n$ for $n\in\mathbb{N}_0$, where the falling factorial $\langle v\rangle_n$ is defined by~\eqref{Fall-Factorial-Dfn-Eq}. The integral representation~\eqref{falling-2ndbER} was obtained in~\cite[p.~12]{Davis-1955-Monthly} and mentioned in~\cite[p.~1]{Nemes-JIS-2011}.
\par
By the way, Theorem~1 in the paper~\cite{Bernoulli2nd-Property.tex} reads that the Bernoulli numbers of the second kind $b_n$ can be represented as
\begin{equation*}%\label{Bernoulli-2rd-int}
b_n=(-1)^{n+1}\int_1^\infty\frac{1}{([\ln(t-1)]^2+\pi^2)t^{n}}\td t, \quad n\in\mathbb{N}.
\end{equation*}
\par
Applying the relation
\begin{equation*}%\label{falling-stirling-relat-eq}
\langle z\rangle_n=\sum_{\ell=0}^{n}s(n,\ell)z^\ell, \quad z\in\mathbb{C}, \quad n\in\mathbb{N}_0
\end{equation*}
in~\cite[p.~9, (1.26)]{Temme-96-book} to the integral representation~\eqref{falling-2ndbER} recovers the closed-form formula
\begin{equation}\label{Nemes-1st-Bernoulli}
b_n=\frac{1}{n!}\int_0^1\sum_{\ell=0}^{n}s(n,\ell)v^\ell\td v
=\frac{1}{n!}\sum_{\ell=0}^{n}\frac{s(n,\ell)}{\ell+1}, \quad n\in\mathbb{N}_0.
\end{equation}
This closed-form formula was derived in~\cite[p.~2]{Nemes-JIS-2011} and mentioned in~\cite[p.~243]{Bernoulli-Stirling-Beograd.tex}.
\end{rem}

\begin{rem}
The identity~\eqref{Conn-Stirl-1st-2nd} is the case $r=-\frac{1}{2}$ in~\eqref{Stirl-1st-2nd-conn}.
\end{rem}

\begin{thm}\label{log-exp-trans-thm}
For $|x|<1$ and $r\in\mathbb{R}$, the series expansion
\begin{multline}\label{log-ser-2stirl-eq}
\biggl[\frac{\ln(1+x)}{x}\biggr]^r=\sum_{n=0}^{\infty}\Biggl[\sum_{k=0}^{n}k!s(n,k) \sum_{m=0}^{k} \frac{(r)_{m}}{(k+m)!} \sum_{\ell=0}^m(-1)^{\ell}\binom{k+m}{m-\ell}S(k+\ell,\ell)\Biggr]\frac{x^n}{n!}
\end{multline}
is valid. Consequently, the combinatorial identity
\begin{multline}\label{log-exp-2stir-Eq}
\sum_{k=0}^n\frac{(r)_k}{k!} \sum_{m=0}^k(-1)^m\binom{k}{m}\frac{s(n+m,m)}{\binom{n+m}{m}}\\
=\sum_{k=0}^{n}k!s(n,k) \sum_{m=0}^{k} \frac{(-r)_{m}}{(k+m)!} \sum_{\ell=0}^m(-1)^{\ell}\binom{k+m}{m-\ell}S(k+\ell,\ell)
\end{multline}
holds for $n\in\mathbb{N}_0$ and $r\in\mathbb{R}$.
\end{thm}

\begin{proof}
Since
\begin{equation}\label{int-repres-exp-log-2}
\frac{\ln(1+x)}{x}=\frac{w(x)}{\te^{w(x)}-1}\quad \text{and}\quad w=w(x)=\ln(1+x),
\end{equation}
we acquire
\begin{gather*}
\biggl(\biggl[\frac{\ln(1+x)}{x}\biggr]^r\biggr)^{(n)}
=\frac{\operatorname{d}^n}{\td x^n}\biggl(\biggl[\frac{w(x)}{\te^{w(x)}-1}\biggr]^r\biggr)\\
=\sum_{k=0}^{n}\biggl(\biggl[\frac{w}{\te^{w}-1}\biggr]^r\biggr)^{(k)} \bell_{n,k}\bigl(w'(x), w''(x), \dotsc,w^{(n-k+1)}(x)\bigr)\\
\to\sum_{k=0}^{n}\lim_{w\to0}\biggl[\biggl(\frac{w}{\te^{w}-1}\biggr)^r\biggr]^{(k)} \lim_{x\to0}\bell_{n,k}\biggl(\frac{1}{1+x}, -\frac{1}{(1+x)^2}, \\
\dotsc,(-1)^{n-k}\frac{(n-k)!}{(1+x)^{n-k+1}}\biggr), \quad x\to0\\
=\sum_{k=0}^{n}k!\Biggl[\sum_{m=0}^{k} \frac{(r)_{m}}{(k+m)!} \sum_{\ell=0}^m(-1)^{\ell}\binom{k+m}{m-\ell}S(k+\ell,\ell)\Biggr]\\
\times\bell_{n,k}\bigl(0!,-1!, \dotsc, (-1)^{n-k}(n-k)!\bigr)\\
=\sum_{k=0}^{n}k!\Biggl[\sum_{m=0}^{k} \frac{(r)_{m}}{(k+m)!} \sum_{\ell=0}^m(-1)^{\ell}\binom{k+m}{m-\ell}S(k+\ell,\ell)\Biggr]s(n,k),
\end{gather*}
where we used the series expansion~\eqref{Bern-Exp-Ser-Gen-Eq} in Theorem~\ref{Bern-Exp-Ser-Gen-thm} and employed the identity~\eqref{Bell(n-k)} and~\eqref{Bell=0!s(n-k)}. This means that the Maclaurin series expansion~\eqref{log-ser-2stirl-eq} is valid.
\par
Comparing~\eqref{log-ser-2stirl-eq} with~\eqref{real-power-log-Eq} in Theorem~\ref{real-power-log-thm} results in the identity~\eqref{log-exp-2stir-Eq}.
The proof of Theorem~\ref{log-exp-trans-thm} is thus complete.
\end{proof}

\section{Conclusions}
In this paper, we established thirteen Maclaurin series expansions of several functions. These expansions are~\eqref{Helms-variant-ser} in Theorem~\ref{Helms-variant-ser-thm}, \eqref{Konwn-exp-results} in Theorem~\ref{Konwn-exp-results-thm}, \eqref{log-cosh-ser}, \eqref{exp+1-ser-log}, \eqref{log-sinh-x-ser}, \eqref{Bell-exp-results} in Theorem~\ref{Bell-exp-results-thm}, \eqref{Sqrt-log-Eq} in Theorem~\ref{Sqrt-log-thm}, \eqref{real-power-log-Eq} in Theorem~\ref{real-power-log-thm}, \eqref{Bern-Exp-Ser-2nd-Eq} in Theorem~\ref{Bern-Exp-Ser-2nd-thm}, \eqref{Bern-Exp-Ser-Eq} in Theorem~\ref{Bern-Exp-Ser-thm}, \eqref{Bern-Exp-Ser-Gen-Eq} in Theorem~\ref{Bern-Exp-Ser-Gen-thm}, \eqref{exp-log-trans-ser} in Theorem~\ref{exp-log-trans-thm}, and~\eqref{log-ser-2stirl-eq} in Theorem~\ref{log-exp-trans-thm}. These functions include
\begin{equation*}
\ln\frac{\te^x+1}{2}, \quad \ln\frac{\te^x-1}{x}, \quad \ln\cosh x, \quad \ln\frac{\sinh x}{x},\quad \biggl[\frac{\ln(1+x)}{x}\biggr]^r, \quad \biggl(\frac{\te^x-1}{x}\biggr)^r
\end{equation*}
for $r=\pm\frac{1}{2}$ and $r\in\mathbb{R}$.
\par
In this paper, we presented four determinantal expressions~\eqref{Bernoulli-Determin-One}, \eqref{Bernoulli-Determin-two}, \eqref{Bernoulli-Determin-3}, and~\eqref{Bernoulli-Determin-4} for the Bernoulli numbers $B_{2n}$, and then we also derived three recursive relations~\eqref{Bernou-REcurs-Relat-Eq}, \eqref{Bernoulli-Recursive-2}, and~\eqref{Bernoulli-Recursive-T3} for the Bernoulli numbers $B_{2n}$.
\par
In this paper, we found out three closed-form formulas~\eqref{Equal=0-Stirl-Bern}, \eqref{Equal-Stirl-Bern2nd}, and~\eqref{Equal-Stirl-Bern3rd} for the Bernoulli numbers $B_{2n}$ and the generalized Bernoulli numbers $B_n^{(r)}$ in terms of the Stirling numbers of the second kind $S(n,k)$, meanwhile we deduced two combinatorial identities~\eqref{Equal=0-Stirl} and~\eqref{Equal-Stirl-Bern2nd=0} for the Stirling numbers of the second kind $S(n,k)$.
\par
In this paper, we acquired two combinatorial identities~\eqref{diag-1st-stirl-eq} and~\eqref{diag-2nd-stirl-eq}, which can be regarded as diagonal recursive relations, for the Stirling numbers of the first and second kinds $s(n,k)$ and $S(n,k)$.
\par
In this paper, we recovered the integral representation~\eqref{falling-2ndbER} and the closed-form formula~\eqref{Nemes-1st-Bernoulli}, and gave an alternative explicit and closed-form formula~\eqref{real-power-log-r=-1} for the Bernoulli numbers of the second kind $b_n$ in terms of the Stirling numbers of the first kind $s(n,k)$.
\par
In this paper, we obtained the identities~\eqref{Conn-Stirl-1st-2nd}, \eqref{Stirl-1st-2nd-conn}, and~\eqref{log-exp-2stir-Eq} connecting the Stirling numbers of the first and second kinds $s(n,k)$ and $S(n,k)$.
\par
The most highlights of this paper include the unification of the generating functions of the Bernoulli numbers $B_n$ and the Stirling numbers of the second kind $S(n,k)$, the unification of the generating functions of the Bernoulli numbers of the second kind $b_n$ and the Stirling numbers of the first kind $s(n,k)$, and the disclosure of the transformations through the relations in~\eqref{int-repres-exp-log} and~\eqref{int-repres-exp-log-2} between these two unifications.

\section{Declarations}

\paragraph{\bf Acknowledgements}
Not applicable.

\paragraph{\bf Funding}
The author was partially supported by the Youth Project of Hulunbuir City for Basic Research and Applied Basic Research.

\paragraph{\bf Institutional Review Board Statement}
Not applicable.

\paragraph{\bf Informed Consent Statement}
Not applicable.

\paragraph{\bf Ethical Approval}
The conducted research is not related to either human or animal use.

\paragraph{\bf Availability of Data and Material}
Data sharing is not applicable to this article as no new data were created or analyzed in this study.

\paragraph{\bf Competing Interests}
The author declares that he has no any conflict of competing interests.

\paragraph{\bf Use of AI tools declaration}
The author declares he has not used Artificial Intelligence (AI) tools in the creation of this article.

\end{document}